\documentclass[leqno]{amsart}

\input xy
\xyoption{all}
\usepackage{amsmath,amsthm,amssymb,setspace}
\usepackage{verbatim} %Needed for the comment environment

\def\TKK{\mbox{\rm TKK}}
\def\uTKK{\widehat{\mbox{\rm TKK}}}

\def\Ldecomp{L_{-1}\oplus L_{0}\oplus L_{1}} % for decomposition of 3-graded Lie alg L
\def\alphadecomp{\alpha_{-1}+\alpha_{0}+\alpha_{1}} % for decomposition of 3-graded hom alpha
\def\Lthr{{\mathcal{LA}}_{\mbox{\rm 3-gr}}} % thr for 3
\def\Lthrinv{{\mathcal{LA}}_{\mbox{\rm 3-gr}}^\varepsilon} % thr for 3 and inv for involution
\def\LAone{{\mathcal{LA}}_{{\mbox{A}_1-\mbox{gr}}}}
\def\sltwo{\mathfrak{sl}_2}
\def\ad{\mbox{\rm ad}}
\def\prP{(P_{-}, P_{+})} % for pair P written coordinate-wise
\def\prPop{(P_{+}, P_{-})} % for pair P^op written coordinate-wise
\def\JP{\mathcal{J\!P}}
\def\JPinvol{\mathcal{J\!P}^\varepsilon} % inv for involution
\def\JTS{\mathcal{J\!T\!S}}

\def\JA{\mathcal{J\!A}}
\def\JAid{\mathcal{J\!A}}
\def\JPinvert{\mathcal{J\!P}^{\rm inv}}
\def\Jmult{}
\def\prJTS{{\mathcal P}_{\mathcal T}} % pr for pair
\def\dbldT{(T, T)} % dbld for "doubled"
\def\ex{\kappa}
\def\dbldS{(S, S)} % dbld for "doubled"
\def\prJA{{\mathcal P}_{\mathcal J\!A}} % pr for pair
\def\dbldJ{(J, J)} % dbld for "doubled"
\def\forJP{{\mathcal F}_{\mathcal P}} % for for "forgetful"
\def\forJTS{{\mathcal F}_{\mathcal T}} % for for "forgetful"
\def\forJA{{\mathcal F}_{\mathcal J\!A}} % for for "forgetful"
\def\gammaP{\gamma}
\def\gammaT{\gamma}
\def\gammaJA{\gamma}
\def\Z{{\Bbb Z}}

\def\Hom{\mbox{Hom}}
\def\Kr{\mbox{Ker}}

\def\ins{{\mathfrak{instr}}}
\def\span{{\rm span}}
\def\Cen{\mbox{\rm Center}}

\def\U{{\mathfrak{A}}}

\def\frs{{\mathfrak s}}

\def\inder{{\mathfrak{inder}}}
\def\id{\mbox{\rm id}}

\newcommand{\beq}{\begin{equation}}
\newcommand{\eeq}{\end{equation}}

\newtheorem*{theorema}{Theorem A}
\newtheorem*{theoremb}{Theorem B}
\newtheorem*{theoremc}{Theorem C}
\newtheorem{lm}{Lemma}[section]

\newtheorem{thm}[lm]{Theorem}

\newtheorem{rem}[lm]{Remark}

\newtheorem{cor}[lm]{Corollary}

\newfam\msbfam
\font\tenmsb=msbm10 \textfont\msbfam=\tenmsb \font\sevenmsb=msbm7
\scriptfont\msbfam=\sevenmsb \font\fivemsb=msbm7
\scriptscriptfont\msbfam=\fivemsb
\def\Bbb{\fam\msbfam \tenmsb}

\def\sset {\subseteq}

\def\to{\rightarrow}

\def\bZ{{\mathbb Z}} % The integers

\title{Categories of Jordan Structures and Graded Lie Algebras}
\author{D.\,M.\,Caveny and O.\,N.\,Smirnov}
\address{Department of Mathematics\\
College of Charleston\\Charleston SC 29424\\U.S.A.}
\date{April 6, 2011}

\begin{document}
\vspace{-1.2in}
\maketitle
%\iffalse%%%%%%%%%%%%%%%%%%%%%%%%%%%%%%%%%%%%%%%%%%%%%%%%%%%%%%%%%%%%%%%%%%%%%%%%%%%%%%%%%%%%%%%%%%%%%%%%%%%%%
\begin{abstract}
In the paper we describe the subcategory of the category of $\Z$-graded Lie algebras which is equivalent to the category of Jordan pairs via a functorial modification of the TKK construction. For instance, we prove that $L = L_{-1} \oplus L_0 \oplus L_1$ can be constructed from a Jordan pair if and only if $L_0=[L_{-1},L_1]$ and the second graded homology group $H_2^{\rm gr} (L)$ is trivial. Similar descriptions are obtained for Jordan triple systems and Jordan algebras. New functorial versions of the TKK construction are given for pairs and algebras.
\end{abstract}

\section{Preliminaries.}
Strong connections between Lie and Jordan algebras were discovered in
the 1960s by Tits, Kantor, and Koecher. Independently and almost
simultaneously they introduced three versions of a construction, known
presently as TKK. Its importance was recognized immediately, as TKK
allows a beautiful and fruitful interplay between the theory of Lie
algebras and the theory of Jordan algebras~\cite{M}. At the same time
the original  TKK construction does not provide a functor for the
categories at hand.

There are two ways to address this. The first is to restrict the
considered objects or morphisms as was done by Koecher and Kantor and, more
recently, by Kac in~\cite{K} and Bertram in~\cite{B}. The second is to
modify the construction. These modifications have been available for
some time even in more general settings (see~\cite{AG} or~\cite{BZ}).
One can also arrive at a functorial modification of TKK by considering
the universal central extension of TKK as was done in~\cite{BN}
or~\cite{T}. It appears, however, that the properties of the functor
have never been studied. These studies are the main goal of our paper.

The categories which naturally arise from the modified construction are
the category of $\Z$-graded Lie algebras and related categories, on one
hand, and the categories of Jordan algebras, triple systems, and pairs
on the other. All three functors which appear from the modified TKK are
full and faithful and are left adjoints of natural forgetful functors.
The main result of our paper contains descriptions of
the image of the functor from Jordan pairs to graded Lie algebras in
terms of graded homology and cohomology groups.

\begin{theorema}
The category of Jordan pairs is equivalent to the category of $3$-graded
Lie algebras $L = L_{-1} \oplus L_0 \oplus L_1$ such that
$L_0=[L_{-1},L_1]$ and $L$ satisfies one of the equivalent conditions:
\begin{itemize}
\item[(i)]  $H_2^{\rm gr} (L) = 0$
\item[(ii)] $H^2 _{\rm gr} (L,M) = 0$ for every module $M$ with the
    trivial grading $M = M_0$.
\end{itemize}
\end{theorema}

As a corollary we have similar results for Jordan triple systems and
Jordan algebras. The latter description has an interesting feature as it
involves regular rather than graded homology and cohomology groups.

\begin{theoremc}
The category of unital Jordan algebras is equivalent to the category of
A$_1$-graded Lie algebras $L$ satisfying one of the equivalent
conditions:
\begin{itemize}
\item[(i)]  $H_2(L) = 0$
\item[(ii)] $H^2(L,M) = 0$ for every trivial module $M$.
\end{itemize}
\end{theoremc}

This paper is organized as follows. The categories of Jordan objects and
corresponding categories of graded Lie algebras are described in
Sections~\ref{Jcats} and~\ref{Lcats}. We recall the original TKK
construction and its versions for pairs and triples in
Section~\ref{TKK}. In Section~\ref{univ}, we introduce a modified
version of TKK and prove its universal properties. The final section
contains various descriptions of images of the functors.

The part of the present paper related to Jordan pairs is parallel to the content from~\cite{S}. It is a 3-graded version of the theory of central extensions of Lie algebras (e.g., see~\cite[Sect. 7.9]{W}) while the latter is a  $\Z_2$-graded version of that theory. Moreover, since the category of Jordan pairs is
equivalent to the category of polarized Lie triple systems (see~\cite[Sect.III.3]{B}), one can apply the results on Lie triple systems from~\cite{S} to invoke our results on Jordan pairs. We found, however, that the direct approach suitably modified works better than application of Lie triple systems results providing shorter arguments that are also self-contained. On the other hand the content related to triple systems and algebras has no analogs in~\cite{S}.

All of our objects (algebras, vector spaces, and modules) will be
defined over a unital commutative ring $k$ that has no $2$- or
$3$-torsion. In the text we often consider pairs of spaces/modules
indexed by $\pm$. Throughout, $\sigma$ stands for an element from
$\{-,+\}$ with $-\sigma$ defined naturally.

%\iffalse%%%%%%%%%%%%%%%%%%%%%%%%%%%%%%%%%%%%%%%%%%%%%%%%%%%%%%%%%%%%%%%%%%%%%%%%%%%%%%%%
\section{Categories of Jordan objects} \label{Jcats}

In this section we describe three categories of Jordan objects which arise naturally in the context of the TKK construction.

\subsection{Jordan Pairs} \label{Jpairs}
A {\it (linear) Jordan pair} is a pair of $k$-modules $P=\prP$ with two
trilinear maps $\{\ ,\ ,\ \}_{\sigma}:P_{\sigma}\times P_{-\sigma}\times P_{\sigma}\to
P_{\sigma}$ such that for $a,c\in P_{\sigma}$ and $b,d\in P_{-\sigma}$, one has
\beq
\{a,b,c\}_\sigma =\{c,b,a\}_\sigma \mbox{ \ and \ } [V_{a,b},V_{c,d}]=V_{V_{a,b}c,d}- V_{c,V_{b,a}d} \label{jp}
\eeq
for $V_{a,b}c:=\{a,b,c\}_\sigma$. When it is clear which trilinear map
applies, we often drop the subscript.  For two Jordan pairs $P=\prP$ and
$Q=(Q_{-},Q_{+})$, a pair of linear maps $\gamma=(\gamma_-,\gamma_+)$,
$\gamma_{\sigma}:P_{\sigma}\to Q_{\sigma}$, is called a {\it Jordan pair
homomorphism} provided
$\gamma_{\sigma}(\{a,b,c\})=\{\gamma_{\sigma}(a),\gamma_{-\sigma}(b),
\gamma_{\sigma}(c)\}$ for $a,c\in P_{\sigma}$ and $b\in P_{-\sigma}$.
We denote the category of all Jordan pairs and their homomorphisms by
$\JP$.

Other Jordan objects, namely Jordan triple systems and Jordan algebras,
can be viewed as Jordan pairs with additional structure.

\subsection{Jordan Triple Systems as Jordan Pairs} \label{Jts}
A {\it Jordan triple system} is a $k$-module $T$ together with a
trilinear map $\{\ ,\ ,\ \}:T\times T\times T\to T$ satisfying
(\ref{jp}) for $a,b,c,d \in T$.  A homomorphism of two
Jordan triple systems $T$ and $S$ is a linear map $\gamma:T\to S$
satisfying the equation
$\gamma(\{a,b,c\})=\{\gamma(a),\gamma(b),\gamma(c)\}$ for $a,b,c\in T$.

The category $\JTS$ of all Jordan triple systems with their
homomorphisms is equivalent to the category $\JPinvol$ of Jordan pairs
with involutions which is described below (see~\cite[Sect.1.13]{L}).

For a Jordan pair $P=\prP$ with the trilinear products $\{ \ , \ , \
\}_{\pm}$, the pair of spaces $P^{\rm op}=\prPop$ with the products $\{ \
, \ , \ \}^{\rm op}_{\pm}:= \{ \ , \ , \ \}_{\mp}$ is also a Jordan
pair, termed the pair {\it opposite} to $P$.
A Jordan pair homomorphism $\varepsilon:P\to P^{\rm op}$ is called an
{\it involution of $P$} provided $\varepsilon_{-\sigma} \circ
\varepsilon _{\sigma} =\mbox{\rm id}_{P_\sigma}$ for $\sigma=\pm$. The
objects of the category $\JPinvol$ are  pairs
$(P,\varepsilon_{\mbox{\it\tiny P}})$ consisting of a Jordan pair $P$
and its involution $\varepsilon_{\mbox{\it\tiny P}}$.

If $(P,\varepsilon_{\mbox{\it\tiny P}})$ and
$(Q,\varepsilon_{\mbox{\it\tiny Q}})$ are two Jordan pairs with
involutions, a homomorphism $\gamma:P\to Q$ is called {\it involutary}
if it  preserves the involutions, i.e., if $\varepsilon_{\mbox{\it\tiny
Q}} \gamma=\gamma ^{\rm op} \varepsilon_{\mbox{\it\tiny P}}$, where
$\gamma ^{\rm op}: P^{\rm op} \to Q^{\rm op}$ is simply $\gamma^{\rm op}
= (\gamma_{+}, \gamma_{-})$. The involutary homomorphisms constitute the
collection of morphisms in $\JPinvol$.

Finally, an equivalence between the categories $\JTS$ and $\JPinvol$ is
provided by the functor $\prJTS$ which sends a triple system $T$ with
the product $\{ \ , \ , \ \} $ to the pair $\dbldT$ with the products
$\{ \ , \ , \ \}_{\pm} := \{ \ , \ , \ \}$ and with the {\it canonical
involution} $\ex_{\pm}={\rm id}_T$. A Jordan triple system homomorphism
$\gamma:T \to S$ is sent to the involutary Jordan pair homomorphism
$\prJTS \gamma=(\gamma,\gamma): \dbldT \to \dbldS$.

\subsection{Jordan Algebras as Jordan Pairs} \label{Ja}

A {\it (linear) Jordan algebra} is a $k$-module $J$ with a bilinear
product $a \Jmult b$ satisfying the following identities
\beq
a \Jmult b = b \Jmult a \quad \mbox{ and } \quad (a^2 \Jmult b) \Jmult a
= a^2 \Jmult (b \Jmult a). \label{JA1}
\eeq
We will work with the category of unital Jordan algebras, where the
algebras as well as homomorphisms are unital.  We denote this category by
$\JAid$. It is equivalent to the category $\JPinvert$ of Jordan pairs
with fixed invertible elements, defined in~\cite[Sect.1.10]{L} as follows.

For a Jordan pair $P=\prP$ an element $b \in P_{\sigma}$ is called {\it
invertible} if the map $P_{-\sigma} \to P_{\sigma}$ defined by $a
\mapsto \{b,a,b\}_{\sigma}$ is invertible. An object of $\JPinvert$ is a
pair $(P,b)$ where $P$ is a Jordan pair and $b$ is an invertible element
from $P_+$. A morphism $\varphi:(P,b)\to (Q,d)$ is defined to be a
Jordan pair homomorphism from $P$ to $Q$ such that $\varphi(b)=d$.

To describe an equivalence functor $\prJA$ from $\JAid$ onto $\JPinvert$
we note that any Jordan algebra $J$ considered with a {\it triple
product} $\{ a, b, c \} := (a \Jmult b) \Jmult c + a \Jmult (b \Jmult c) - b
\Jmult (c \Jmult a) \label{a2ts}$ is a Jordan triple system and hence $\prJTS J=\dbldJ$ is a Jordan pair.
Moreover, the identity element $1\in J=(\prJTS J)_+$ is an invertible
element of the pair. The functor $\prJA$ sends $J$ to $(\dbldJ,1)$ and
the Jordan algebra homomorphism $\gamma$ to the pair homomorphism $\prJA
\gamma = (\gamma , \gamma)$.

\section{Categories of Lie Algebras}\label{Lcats}

In this section we review the categories of Lie algebras related to the
aforementioned Jordan structures.

\subsection{3-graded Lie Algebras and Jordan Pairs} \label{3gr}

Let $L=\bigoplus_{i\in \Z} L_{i}$ be a $\Z$-graded Lie algebra, that is
$L$ is a direct sum of indexed $k$-modules $\{L_i:i\in \Z\}$ such that
$[L_i,L_j]\sset L_{i+j}$. If $L_i=0$ for $|i|>1$, we write $L=\Ldecomp$
and say that the algebra $L$ is {\it $3$-graded}.
We will work with the category $\Lthr$ of all $3$-graded Lie
algebras and their graded homomorphisms. For every such morphism
$\alpha:L \to K$ we have $\alpha (L_i) \sset K_i$ for $i \in \{0,
\pm1\}$ and we write $\alpha=\alphadecomp$ where
$\alpha_i=\alpha|_{L_i}$.
It is well-known that for a $3$-graded Lie algebra $L = \Ldecomp$ the
pair $(L_{-1}, L_1)$ is Jordan with respect to the operations $\{a, b,
c\}_\sigma := [[a,b],c]$, with $a,c \in L_{\sigma 1}$ and $b \in
L_{-\sigma 1}$. It is also clear that for any graded Lie homomorphism
$\alpha$ the pair $(\alpha_{-1},\alpha_1)$ is a Jordan homomorphism. In
other words, we have a forgetful functor $\forJP$ from the category of
3-graded Lie algebras $\Lthr$ to the category of Jordan pairs $\JP$.

\subsection{Involutions on 3-graded Lie Algebras and Jordan Triple
Systems} \label{3grInv}

An {\it anti-graded involution} of a $\bZ$-graded Lie algebra $L$ is a
homomorphism $\varepsilon:L \to L$ satisfying $\varepsilon^2 = \id _L$
and $\varepsilon (L_i) \sset L_{-i}$ for $i \in \bZ$.  For $\bZ$-graded
Lie algebras with fixed involutions, $(L,\varepsilon_L)$ and
$(K,\varepsilon_K)$, we say that a homomorphism $\alpha:L \to K$ is {\it
involutary} if it preserves the involutions; that is, if $\alpha
\varepsilon_L = \varepsilon_K \alpha$.  We let $\Lthrinv$ denote the
category whose objects are $3$-graded Lie algebras with fixed
anti-graded involution and whose morphisms are graded involutary
homomorphisms.

If $\varepsilon$ is an anti-graded involution of a 3-graded algebra
$L=\Ldecomp$, then the restriction $\widetilde\varepsilon$ of
$\varepsilon$ onto the pair $(L_{-1}, L_1)$ is a Jordan involution. It
is easy to see that in this case, the Jordan pair with involution
$(\forJP L,\widetilde\varepsilon)$ is isomorphic to the pair $\prJTS
L_1$ via the map $(a,b)\mapsto (\varepsilon(a),b)$ where
$L_1$ is considered as Jordan triple system with the
operation $\{a,b,c\}:= [[a, \varepsilon(b)], c]$. This leads to a
forgetful functor $\forJTS$ from $\Lthrinv$ to the category of Jordan
triple systems $\JTS$. This forgetful functor takes involutary graded
homomorphisms $\alpha$ to Jordan triple system homomorphisms $\forJTS
\alpha = \alpha_{1}$.

\subsection{A$_1$-graded Lie Algebras and Jordan Algebras} \label{a1gr}
The third category of Lie algebras is part of the theory of Lie algebras
graded by root systems initiated by Berman and Moody in~\cite{BM}. The
general definition of root-graded algebra over a field along with
classification theorems can be found in~\cite{BM}, \cite{BZ},
and~\cite{ABG2}. Here, we give its adaptation for the case when the root
system is A$_1$ and $k$ is a commutative ring.

Following~\cite{K}, we say that an ordered triple $\frs=\langle h, e,
f\rangle$ of elements of a Lie algebra is an $\sltwo$-{\it triple} if
one has the following commutator relations:
$$
[e,f]=h,\quad [h,e]=2e, \mbox{ and } [h,f]=-2f.
$$
An A$_1$-{\it graded} Lie algebra is a pair $(L,\frs)$ that consists of
a Lie algebra $L$ and an $\sltwo$-triple $\frs\subseteq L$ such that $L
= \Ldecomp$ where $L_i = \{x \in L: [h,x]=2ix\}$ for $i \in \{0, \pm
1\}$ and $L_0=[L_{-1},L_1]$. All $A_1$-graded Lie algebras together with
Lie algebra homomorphisms sending $\sltwo$-triple to $\sltwo$-triple
element-wise form a category which we denote by $\LAone$.

To relate this category to the previous ones we note that for an
A$_1$-graded algebra $(L,\frs)$ the decomposition $L = \Ldecomp$ is a
3-grading and that $e\in L_1$ is an invertible element of the Jordan
pair $\forJP L$. Thus the pair $\forJP L$ is obtained from the Jordan
algebra $L_1$ considered with operation $a \Jmult b :=[[a,f],b]$
and $\frac{1}{2} e$ is the identity element of this algebra. Further,
for any morphism $\alpha$ of A$_1$-graded algebras, the restriction
$\alpha_1=\alpha|_{L_1}$ is a homomorphism of unital Jordan algebras.
Thus one has the third forgetful functor $\forJA$ from $\LAone$ to the
category of unital Jordan algebras $\JAid$.

%\fi%%%%%%%%%%%%%%%%%%%%%%%%%%%%%%%%%%%%%%%%%%%%%%%%%%%%%%%%%%%%%%%%%%%%%%%%%%%%%%%

\section{Tits-Kantor-Koecher construction}\label{TKK}

In this section we recall the constructions of the structure Lie algebra
$\ins (P)$ and of $\TKK (P)$ for a Jordan pair $P$. We also describe
additional structures arising on these Lie algebras when $P$ is
obtained from a Jordan triple system or a Jordan algebra. Throughout the
paper $\mathfrak{gl}(V)$ stands for the Lie algebra of all linear
transformations on a $k$-module space $V$.

\subsection{Inner Structure Algebra}\label{inner} Let $P=\prP$ be a Jordan pair. For
any pair $(a,b)\in P_{-}\times P_{+}$ we define $\nu(a,b)\in
\mathfrak{gl}(P_-)\oplus \mathfrak{gl}(P_+)$ by $\nu (a,b) := (V_{a,b}, -V_{b,a})$.
Identities~(\ref{jp}) imply that
\beq
[\nu(a,b),\nu(c,d)]=\nu(\nu(a,b)c,d)+\nu(c,\nu(a,b)d) \label{nubracket}
\eeq
for $a,c\in P_{-}$ and $b,d\in P_{+}$, so the span of the $\nu(a,b)$'s
forms a subalgebra of the Lie algebra $\mathfrak{gl}(P_-)\oplus
\mathfrak{gl}(P_+)$ called the {\it inner structure algebra} of $P$ and
denoted by $\ins (P)$.

If $P=\prJTS T$ for a Jordan triple system $T$, we write $\ins (T)$ to
denote $\ins (\prJTS T)$. In this case it is easy to see that the map
$\nu(x,y)\mapsto -\nu(y,x)$ is an automorphism of the Lie algebra $\ins (T)$ of period two. Hence
there is a eigenspace decomposition
\beq\label{DecIns1}
\ins (T)=\ins(T)_{-1}\oplus\ins(T)_{1}.
\eeq
Here $\ins(T)_{1}=\span\{\nu (a,b)-\nu (b,a):a,b\in T\}$ is the
subalgebra of {\it inner derivations of} $T$, denoted usually by
$\inder(T)$.

Furthermore, if $P=\prJA J$ for a unital Jordan algebra $J$, then the
component $\ins(J)_{-1}=\span\{\nu (a,a):a\in J\}$ can be identified
with the submodule $R_J$ of operators of right multiplication  because
$\nu(a,a)=(R_{a^2}, -R_{a^2})$. So the decomposition above admits the
form
\beq\label{DecIns2}
\ins (J)=R_J\oplus\inder(J).
\eeq

\subsection{TKK construction} \label{TKKp}  The following theorem
introduces the celebrated Tits-Kantor-Koecher construction.

\begin{thm} For any Jordan pair $P=\prP$, the $k$-module $P_{-}\oplus
\ins (P)\oplus P_{+}$ together with the bracket
\beq\label{tkk}
[a+X+b,c+Y+d]
=(Xc-Ya)+([X,Y]+\nu(a,d)-\nu(c,b))+(Xd-Yb)
\eeq
for $X,Y\in \ins(P)$, $a, c \in P_-$, and $b,d \in P_+$, is a Lie
algebra.
\end{thm}
This Lie algebra is denoted by $\TKK(P)$ and is called the {\it
Tits-Kantor-Koecher construction}, or {\it TKK} for short.  It follows
readily from~(\ref{tkk}) that the algebra $\TKK(P)$ is 3-graded and
$\forJP(\TKK(P))=P$ for the forgetful functor $\forJP$ defined in
Section~\ref{3gr}.

It was noted by Kantor and Koecher that the TKK construction has certain
functorial properties. Specifically, if one allows only epimorphisms as
morphisms in the category of Jordan pairs under consideration, then TKK
constitutes a functor from this category to $\Lthr$.  In general
however, to be able to extend a Jordan pair homomorphism $\gamma:P\to Q$
to a Lie algebra homomorphism $\widehat{\gamma}:\TKK(P)\to \TKK(Q)$, one
needs the condition: $\sum _i V_{a_i,b_i}=0 \mbox{ implies }\sum _i
V_{\gamma(a_i),\gamma(b_i)}=0$.

\subsection{TKK construction for Jordan triple systems and algebras}
\label{TKKts}
 Here our goal is to show that whenever the pair $P$ is obtained from a
 Jordan triple system or a Jordan algebra, the Lie algebra $\TKK(P)$ can
 be considered naturally as an object of $\Lthrinv$ or of $\LAone$,
 respectively.
Assume that $P=\prJTS T$ for a Jordan triple system $T$. Then it is easy
to see that the endomorphism $\bar\kappa$ on $\TKK(P)$ defined by
\beq\label{canI}
\bar\kappa(a_-+\nu(c,d)+b_+)=b_--\nu(d,c)+a_+
\eeq
is an anti-graded involution. The restriction of $\bar\kappa$ onto
$(T,T)$ is the canonical involution $\kappa$ of the pair $\prJTS T$
defined in Sect.~\ref{Jts}, so we can conclude that $\forJTS
\TKK(P)=T$.

Assume now that $P=\prJA J$ for a Jordan algebra $J$ with the identity
element~1.  Then the triple $\bar\frs=\langle -\nu(1,2), 2_+, 1_-\rangle$
 is an $\sltwo$-triple of $\TKK(P)$, the 3-grading defined by this
 triple coincides with the standard grading, and $\forJA \TKK(P)=J$.

%\iffalse%%%%%%%%%%%%%%%%%%%%%%%%%%%%%%%%%%%%%%%%%%%%%%%%%%%%%%%%%%%%%%%%%%%%%%%%%%%%%%%%

\section{Universal Imbeddings of Jordan objects}\label{univ}
In the previous section we saw that any Jordan object can be imbedded
into a Lie algebra of a certain type. It is not difficult to see that
there exists a universal imbedding, even in much more general settings
(see~\cite[Sect.1.4]{BZ}, for example). In the next section we present a
different construction of a universal object. The construction is obtained as a result of a general
procedure that first appeared in~\cite{BS} and was modified in~\cite{S}. We
believe simplicity of the generating set to be another merit of our
construction.

\subsection{Construction of universal imbedding} Let $P=\prP$ be a
Jordan pair and $\ins(P)$ be its structure Lie algebra described in
Sect.~\ref{inner}. Since $\ins(P)$ is a subalgebra of
$\mathfrak{gl}(P_-)\oplus \mathfrak{gl}(P_+)$, it acts canonically on
$P_-$ and $P_+$, and on the tensor product $P_-\otimes P_+$ via $X\big(\sum_ia_i\otimes b_i\big)=\sum_iXa_i \otimes
b_i+\sum_ia_i\otimes  Xb_i$. Furthermore, it follows
from~(\ref{nubracket}) that the map $\lambda : P_-\otimes  P_+ \to
\ins(P)$ defined by $\lambda (a \otimes b) = \nu (a,b)$ is a module
homomorphism from $P_-\otimes P_+$ to the regular module $\ins(P)$.
Therefore we can apply the following

\begin{lm} \label{mod} {\rm\bf(\cite[Lemma 3.1]{S})}
Let $M$ be a module over a Lie algebra $L$ and let $\lambda:M\to L$ be
an $L$-module epimorphism. Then $A(M)=\span\{\lambda(m)\cdot m:m\in M\}$ is a
submodule of $M$, the quotient module $Q=M/A(M)$ with the product $[p,q]=\mu(p)\cdot q$ is a Lie
algebra, and the map $M/A(M)\to L$ induced by $\lambda$ is a central extension.
\end{lm}
 In our case the set $A(P_{-} \otimes P_{+})$  is spanned by elements of the form $\nu(a,b)(a\otimes b)$
or equivalently by
\beq \nu(a,b)(c\otimes d)+\nu(c,d)(a\otimes b) \label{t4}
\eeq
where $a,c\in P_{-}$ and $b,d\in P_{+}$. We let $\langle P_{-}, P_{+}
\rangle$ denote the quotient module $(P_{-} \otimes P_{+})/A(P_{-}
\otimes P_{+})$ and let $\langle a,b\rangle$ denote the coset containing
$a\otimes b$. Lemma~\ref{mod} yields:

\begin{cor} \label{centralinstr}
$\langle P_{-},P_{+}\rangle$ is a Lie algebra relative to the product
$[\langle a,b\rangle,\langle c,d\rangle]= \langle \{a,b,c\},d\rangle-
\langle c,\{b,a,d\}\rangle.$
Moreover, the map $\mu:\langle P_{-},P_{+}\rangle\to \ins(P)$, defined
by $\mu(\langle a,b\rangle)=\nu(a,b)$, is a central extension.
\end{cor}

In what follows, the $\ins(P)$-modules $P_-$ and $P_+$ will be viewed as
$\langle P_-, P_+ \rangle$-modules via $\mu$.

\begin{thm} \label{univTKK} For any Jordan pair $P=\prP$, the $k$-module
$$\uTKK(P)=P_{-}\oplus \langle P_{-},P_{+}\rangle\oplus P_{+}$$ together
with the bracket
\beq\label{utkk}
[a+X+b,c+Y+d]
=(Xc-Ya)+([X,Y]+\langle a,d\rangle-\langle c,b\rangle)+(Xd-Yb)
\eeq
for $X,Y\in \langle P_{-},P_{+}\rangle$, $a, c \in P_-$, and $b, d \in
P_+$, is a 3-graded Lie algebra and the map
$\upsilon:\uTKK(P)\to \TKK(P)$ defined by
\beq\label{ceTTK}
\upsilon(a+X+b)=a+\mu(X)+b
\eeq
is a graded central extension.

Moreover, in the special case when $P=\prJTS T$ for a Jordan triple
system $T$, the map
\beq\label{ucanI}
\widehat{\kappa}(a_-+\langle c,d\rangle+b_+)=b_--\langle d,c\rangle+a_+
\eeq
is an anti-graded involution of $\uTKK(P)$ and $\upsilon$ is an
involutary morphism from $(\uTKK(P),\widehat\kappa)$ to
$(\TKK(P),\bar\kappa)$.

In the special case when $P=\prJA J$ for a unital Jordan algebra $J$,
the triple $\widehat\frs=\langle -\langle 1,2\rangle, 2_+, 1_-\rangle$
is an $\sltwo$-triple of $\uTKK(P)$, the 3-grading defined by this
triple coincides with the canonical grading, and $\upsilon$ is an
A$_1$-graded morphism from $(\uTKK(P),\widehat\frs)$ to
$(\TKK(P),\bar\frs)$.
\end{thm}

\begin{proof} The fact that $\uTKK(P)=P_{-}\oplus \langle
P_-,P_+\rangle\oplus P_{+}$ is a 3-grading follows directly
from~(\ref{utkk}). We write $|l|=i$ if $l\in \uTKK(P)_i$.

It is easy to see from~(\ref{utkk}) that the bracket $[\,,\,]$ is
anticommutative, and from~(\ref{tkk}) and~(\ref{utkk}) that $\upsilon$
is an epimorphism with $\Kr(\upsilon)\subseteq \uTKK(P)_0=\langle
P_{-},P_{+}\rangle$. It follows that the Jacobian ${\mathfrak
j}(l,k,m)$, defined by
${\mathfrak j}(l,k,m)=[[l,k],m]+[[k,m],l]+[[m,l],k]$, is zero for $l, k,
m \in \uTKK(P)$ with $|l|+|k|+|m|\ne 0$.
Assume that $|l|+|k|+|m|=0$. Because ${\mathfrak j}(l,k,m)$ is an
alternating function and because of the $+/-$ symmetry, it suffices to
check that ${\mathfrak j}(l,k,m)=0$ when $(|l|,|k|,|m|)=(0,-1,1)$. Here
is the calculation for this case:
$$
{\mathfrak j}\big(X,a,b\big)=\langle Xa,b\rangle-
[X, \langle a, b \rangle] +\langle a,Xb\rangle=0.
$$
Thus, $\uTKK(P)$ is a 3-graded Lie algebra and $\upsilon$ is a graded
epimorphism. It is easy to see that
$\Kr(\upsilon)=\Kr(\mu)\subseteq\Cen(\uTKK(P))_0$.

Assume now that $P=\prJTS T$ for a Jordan triple system $T$. It follows
from~(\ref{utkk}) that $\widehat\kappa$ is an anti-graded involution of
$\uTKK(P)$ and it is clear that $\upsilon: (\uTKK(P),\widehat\kappa)
\longrightarrow (\TKK(P),\bar\kappa)$ is a morphism in $\Lthrinv$.

Finally, when $P=\prJA J$ for a unital Jordan algebra $J$, it is easy to
check that $\widehat\frs=\langle \langle 2,1\rangle, 2_+, 1_-\rangle$ is
an $\sltwo$-triple of $\uTKK(P)$, the 3-grading defined by this triple
coincides with the standard grading, and $\upsilon:
(\uTKK(P),\widehat\frs) \longrightarrow (\TKK(P),\bar\frs)$ is a
morphism in $\LAone$.

\end{proof}

Our construction $\uTKK(P)$ has an especially simple form for the case
of Jordan algebras. To this end we need:

\begin{lm}
If $J$ is a Jordan algebra with the identity element $1$, then $\langle
J,J\rangle=J\otimes J/A(J\otimes J)$ where $A(J\otimes J)$ is spanned by
the elements of the form
\beq
a^2\otimes a-1\otimes a^3. \label{Aspan}
\eeq
\end{lm}

\begin{proof} We need to prove that the set $A=A(J\otimes J)$ defined in
Section~5.1 equals to the span of~(\ref{Aspan}) denoted temporarily by
$B$. In the proof we write $x \equiv _T y$ if $x - y$ is an element of a
set $T$. We note that for the linear map $\omega: J \otimes J
\longrightarrow J \otimes J$,  defined by $\omega (a \otimes b) = b
\otimes a$, one has $(\nu_{a,b} (c \otimes d))^\omega =- \nu_{b,a} (d
\otimes c)$ for $a,b,c, d \in J$. Therefore we have $A ^\omega
\subseteq A$ and hence $x \equiv _A y$ implies $x^\omega \equiv _A
y^\omega$.
First we show that $B \subseteq A$.
Specializing to $b=c=d=1$ in~(\ref{t4}) and noting that $\nu_{1,1}$ acts
as zero on $J \otimes J$, we have
\begin{equation}
\ a \otimes 1 \equiv_A 1 \otimes a. \label{swapwitheinA}
\end{equation}
Next we specialize to $b=c=a$ and $d=1$ to obtain
\begin{equation}
a^3\otimes 1-a\otimes a^2\equiv _A a \otimes a^2-a^2 \otimes a.
\label{skewinA}
\end{equation}
Since the right-hand side is skew-symmetric relative to $\omega$,
(\ref{swapwitheinA}) and (\ref{skewinA}) imply that
\begin{equation*}
\begin{split}
2(a^3\otimes 1-a\otimes a^2)\equiv _A (a^3\otimes 1-a\otimes
a^2)-(a^3\otimes 1-a\otimes a^2)^\omega \equiv _A
a^3\otimes 1\\-a\otimes a^2-1\otimes a^3+a^2\otimes a
\equiv _A
 -(a \otimes a^2-a^2 \otimes a)
\equiv _A -(a^3\otimes 1-a\otimes a^2)
\end{split}
\end{equation*}
so $a^3\otimes 1-a\otimes a^2\equiv _A 0$ and $B\subseteq A$.

To show the reverse containment, we will use the linearized version
of~(\ref{Aspan}):  \begin{equation}
ab\otimes c +bc\otimes a + ac\otimes b\equiv _B 1\otimes (ab)c+1\otimes
(bc)a+1\otimes (ca)b. \label{master}
\end{equation}
Setting $b=c=1$ in this equation results in
\begin{equation}
a\otimes 1\equiv _B 1\otimes a, \label{b1}
\end{equation}
while setting $c=1$ and then applying~(\ref{b1}) gives us
\begin{equation}
a\otimes b +b\otimes a\equiv _B 2\otimes ab. \label{b2}
\end{equation}
It follows that for every element $X\in J\otimes J$
\begin{equation}
X\equiv_B X^\omega \mbox{ if and only if } X\equiv_B 1\otimes \eta(X),
\label{Beta}
\end{equation}
where $\eta:J\otimes J\to J$ is a linear map defined by $\eta(a\otimes
b)=ab$. Our next goal is to prove that the generators~(\ref{t4}) satisfy
one of the equivalent conditions from~(\ref{Beta}). Let us denote the
set of all such elements by $S^B$.

Specializing $c=ab$ in~(\ref{master}), we obtain $(ab)a\otimes
b+b(ab)\otimes a\equiv_{S^B}0$ and use~(\ref{b2}) to write it as
\begin{equation}
(ab)a\otimes b-a\otimes b(ab)\equiv_{S^B}0.\label{m1}
\end{equation}
Replacing $a$ with $a^2$ in~(\ref{master}) and setting $b=c$ yields
$2a^2b\otimes b\equiv_{S^B}0$. Adding this to $2b^2a\otimes
a\equiv_{S^B}0$ and using~(\ref{b2}) gives us
\begin{equation}
a^2b\otimes b-a\otimes b^2a\equiv_{S^B}0. \label{m2}
\end{equation}
Now (\ref{m1}) and (\ref{m2}) together with~(\ref{a2ts}) imply that a
generator~(\ref{t4}) of $A$ is
\begin{equation*}
%\begin{split}
\{a,b,a\}\otimes b-a\otimes \{b,a,b\}
=2(ab)a\otimes b - a^2b\otimes b -2a\otimes b(ab)+a\otimes
b^2a\equiv_{S^B}0\label{m3}
%\end{split}
\end{equation*}
so it is in $S^B$. However the result of application of $\eta$ to this
element is a special case of the linearized Jordan identity
(see~\cite[Ident.(23), p.86]{ZSSS}):
$$
2((ab)a)b - (a^2b)b -2a(b(ab))+a(b^2a)=0.
$$
Consequently, (\ref{Beta}) implies that any generator (\ref{t4}) of $A$
is in $B$.
\end{proof}

\begin{cor} \label{univTKKforJA} For a unital Jordan algebra $J$, set
$J_+$ and $J_-$ to be two copies of $J$ and $\langle J,J\rangle=J\otimes
J/\span\{a^2\otimes a-1\otimes a^3\}$. Then the $k$-module
$$\uTKK(J)=J_{-}\oplus \langle J,J\rangle\oplus J_{+}$$ together with
the bracket ~(\ref{utkk}) is a 3-graded Lie algebra.
\end{cor}

\begin{rem}{\rm It is also possible to select symmetric/skew-symmetric
tensors generating $A=A(J\otimes J)$. Let
$S^2(J)= \span \{a\otimes a: a \in J \}$ and $J\wedge J= \span
\{a\otimes b-b\otimes a: a,b \in J \}$ be the sets of symmetric and
skew-symmetric tensors and let $A_{symm}=A(J\otimes J)\cap S^2(J)$ and
$A_{skew}=A(J\otimes J)\cap (J\wedge J)$. In the proof we noted that
$A^\omega \subseteq A$, so $A= A_{symm}\oplus A_{skew}$.
Moreover one can prove that $A_{symm}= \span \{ 2 a\otimes a + 1\otimes a^2 +a^2\otimes 1: a \in J
\}$ and $A_{skew}= \span \{a^2\otimes a-a\otimes a^2: a \in J \}$. So one has the following decomposition of the algebra $\langle
J,J\rangle$ \begin{equation*}
\langle J,J\rangle=S^2(J)/A_{symm}\oplus (J\wedge J)/A_{skew}
\end{equation*}
which covers the one in~(\ref{DecIns2}).
}
\end{rem}

\subsection{Universal property of $\uTKK(P)$}
In the theorem below we consider $P$ as a subset of $\uTKK(P)$.

\begin{thm} \label{Tuniv}
Let $L=L_{-1}\oplus L_{0}\oplus L_{1}$ be a 3-graded Lie algebra.
\begin{itemize}
\item[{\rm (i)}] For every Jordan pair homomorphism $\gammaP :P\to
    \forJP(L)$, there is a unique graded Lie algebra homomorphism
    $\widehat\gamma:\uTKK(P)\to L$ extending $\gammaP$, that is
    $\widehat\gamma_{\pm 1}=\gammaP_{\pm}$.
It is defined by the formula
\beq\label{ext}
\widehat\gamma(a+\sum_i\langle c_i,d_i\rangle
+b)=\gammaP_{-}(a)+\sum_i[\gammaP_{-}(c_i),\gammaP_{+}(d_i)]+\gammaP_{+}(b)
\eeq
for $a, c_i \in P_-$ and $b, d_i \in P_+$.
\item[{\rm (ii)}] If $L$ has an anti-graded involution
    $\varepsilon$, then for every Jordan triple system homomorphism
    $\gammaT: T \to \forJTS(L,\varepsilon)$ there is a unique
    involutary Lie algebra homomorphism
    $\widehat\gamma:(\uTKK(\prJTS T),\widehat \ex) \to
    (L,\varepsilon)$ extending $\gammaT$, where $\widehat\ex$ is the
    canonical anti-graded involution on $\uTKK (\prJTS T)$.
\item[{\rm (iii)}] Furthermore, if $(L,\frs)$ is an $A_1$-graded Lie
    algebra, then for every unital Jordan algebra homomorphism
    $\gammaJA :J\to \forJA (L,\frs)$, there is a unique A$_1$-graded
    morphism $\widehat\gamma:(\uTKK(\prJA J),\widehat\frs)\to
    (L,\frs)$ extending $\gammaJA$.
\end{itemize}
\end{thm}

\begin{proof} For a pair homomorphism $\gammaP: P \longrightarrow \forJP
(L)$, we begin by constructing $\widehat\gamma_0:\langle
P_-,P_+\rangle\to L_0$.  It is easy to see that the set $A(P_{-}\otimes
P_{+})$ is in the kernel of the linear map $\varphi : P_{-}\otimes
P_{+}\to L_{0}$ defined by $\varphi(a\otimes
b)=[\gamma_{-}(a),\gamma_{+}(b)]$. Thus one can consider the map
$\widehat\gamma_0$ induced by $\varphi$ on $\langle
P_-,P_+\rangle=(P_{-}\otimes P_{+})/A(P_{-}\otimes P_{+})$.

We let $\widehat\gamma:=\gamma_-+\widehat\gamma_0+\gamma_+:P_{-}\oplus
\langle P_-,P_+\rangle\oplus P_{+}\to \Ldecomp$, which is clearly
graded.
It follows from the definition of $\widehat\gamma$ that
$\widehat\gamma[a,b]=[\gamma(a),\gamma(b)]=[\widehat\gamma(a),
\widehat\gamma(b)]$. Furthermore, for $a,c\in P_-$ and $b,d\in P_+$ one
has
$\widehat\gamma[\langle a,b\rangle,c]=\gamma\{a,b,c\}=
[[\gamma(a),\gamma(b)],\gamma(c)]=
[\widehat\gamma\langle a,b\rangle,\widehat\gamma(c)]$ and similarly
$\widehat\gamma[\langle a,b\rangle,d]=[\widehat\gamma\langle
a,b\rangle,\widehat\gamma(d)]$. Since the subspaces $P_-$ and $P_+$
generate the algebra $\uTKK(P)$, it follows that $\widehat\gamma$ is a
Lie algebra homomorphism and a unique extension of $\gammaP$.

To establish assertion (ii), we consider an anti-graded involution
$\varepsilon$ on $L$ and a triple system homomorphism $\gammaT: T \to
\forJTS(L,\varepsilon)$.  We lift $\gammaT$ to the Jordan pair
homomorphism $(\varepsilon \gammaT, \gammaT): \dbldT \to \forJP (L)$.
As previously, we define a linear map $\varphi: T \otimes T \to L_0$ by
$\varphi( a \otimes b):= [\varepsilon \gammaT(a), \gammaT (b)]$, which
gives us a well-defined map $\widehat \gamma _0$ on $\langle T, T
\rangle$ and a unique Lie algebra homomorphism $\widehat \gamma :=
\varepsilon \gammaT + \widehat \gamma _0 + \gammaT$.  Furthermore,
$\widehat \gamma: (\uTKK(\prJTS T),\widehat \kappa) \to (L,\varepsilon)$
is involutary.

We proceed similarly with assertion (iii). Assume that $(L,\frs)$ is an
object in $\LAone$ for an $\sltwo$-triple $\frs = \langle h, e, f
\rangle$ and that $\gammaJA :(J,1)\to \forJA (L,\frs)$ is a unital
Jordan algebra homomorphism.  We lift $\gammaJA$ to a Jordan pair
homomorphism $(\varepsilon \gammaJA, \gammaJA): \dbldJ \to \forJP (L)$
just as in the proof of assertion (ii), using the anti-graded involution
$\varepsilon = (\varepsilon _{-}, \varepsilon _{+})$ on $L$ given by $(-
\frac{1}{2} (\ad e)^2, - \frac{1}{2} (\ad f)^2)$.  Then we define the
unique Lie algebra homomorphism $\widehat \gamma$ as before.  It is easy
to see that this $\widehat \gamma$ takes the $\sltwo$-triple $\langle
\langle 2, 1 \rangle, 2_{+}, 1 _{-} \rangle$ in $\uTKK(\prJA J)$
element-wise to the triple $\langle h, e, f \rangle$ in $L$.
\end{proof}
%\fi%%%%%%%%%%%%%%%%%%%%%%%%%%%%%%%%%%%%%%%%%%%%%%%%%%%%%%%%%%%%%%%%%%%%%%%%%%%%%%%%%%%%%%%%%%%%%%%%%%%%%%%%%%

\section{Categories of Lie algebras equivalent to $\JP$, $\JTS$, and
$\JA$}

\subsection{$\uTKK$ as left adjoint} Part (i) of Theorem~\ref{Tuniv}
implies that the identity map $P \to \forJP \uTKK(P)$ is a
universal arrow from $P$ to $\forJP$ (see~\cite{ML} for definitions and properties).
Hence,
the construction $\uTKK$ constitutes a left adjoint to the forgetful
functor $\forJP:\Lthr\to \JP$ from Section~\ref{3gr}.
Similarly, it follows from (ii) and (iii) of the theorem that functors
$\uTKK:\JTS\to\Lthrinv$ and $\uTKK:\JA\to\LAone$ are left adjoints to
the forgetful functors $\forJTS$ and $\forJA$ described in~Sections
\ref{3grInv} and~\ref{a1gr}, respectively. We do not introduce
additional notations for these functors since it will be clear from the
context which functor is considered.

In addition, formula~(\ref{ext}) implies that all three of these
adjoints are full and faithful. It follows that $\uTKK$ provides an
equivalence between the categories of Jordan objects $\JP$, $\JTS$,
$\JA$ on the one hand and full subcategories of $\Lthr$, $\Lthrinv$, and
$\LAone$, respectively, on the other hand. Our next  goal is to give
intrinsic descriptions of these subcategories.\\
%\fi%%%%%%%%%%%%%%%%%%%%%%%%%%%%%%%%%%%%%%%%%%%%%%%%%%%%%%%%%%%%%%%%%%%%%%%%%%%%%%%%%%%%%%%%%%%%%%%%%%%

\subsection{Images of $\uTKK$} \label{imageTKK}
The model for further development is the theory of central extensions of
perfect Lie algebras (e.g. see~\cite[Sect.\,7.9]{W}) and its $\bZ
_2$-graded version described in~\cite{S}. Here, we provide only a brief version of these standard arguments.

To describe Lie algebras $L$ isomorphic to $\uTKK(P)$, we note first
that definition~(\ref{utkk}) implies the equality $L_{0}=[L_{-1},L_1]$.
So for a 3-graded Lie algebra $L=L_{-1}\oplus L_{0}\oplus L_{1}$, we say
that $L$ is 0-{\it perfect} provided $L_0=[L_{-1},L_1]$. In particular
the graded algebras $\TKK(P)$ and $\uTKK(P)$ are 0-perfect.
Furthermore, we noted in Theorem~\ref{univTKK} that $\uTKK(P)$ is a
central extension of $\TKK(P)$. In fact this extension is universal in
the class of central $0$-extensions. Here we say that a graded
homomorphism $\varphi: K_{-1}\oplus K_{0}\oplus K_{1}\to L_{-1}\oplus
L_{0}\oplus L_{1}$ is a 0-{\it extension} of $L$ if $\varphi$ is
surjective and $\Kr(\varphi)\subseteq K_0$. A central $0$-extension is
said to be {\it universal} if every other central $0$-extension of $L$
factors uniquely through it. A graded algebra is said to be {\it
centrally $0$-closed} if every central $0$-extension of the algebra
splits in a unique way.

\begin{thm} \label{TuTKK} If a 3-graded Lie algebra $L$ is 0-perfect,
then $\upsilon:\uTKK(\forJP L)\to L$ is a universal central 0-extension
of $L$ for the map $\upsilon=\widehat{(\mbox{\rm id}_{\forJP L})}$ as in
Theorem~\ref{Tuniv}.
\end{thm}

\begin{proof}
If $\varphi: K\to L$ is a central 0-extension of $L$ then the Jordan
homomorphism $\forJP \varphi:\forJP K\to \forJP L$ is invertible and the
map $(\forJP \varphi)^{-1}:\forJP L \to\forJP K$ can be extended
uniquely to a Lie algebra homomorphism $\psi:\uTKK(\forJP L )\to K$ by
Theorem~\ref{Tuniv}.  Hence $(\forJP\varphi)\circ(\forJP\psi)=\mbox{\rm
id}_{\forJP L}=\forJP\upsilon$ implies $\varphi\circ\psi=\upsilon$ since
$\uTKK(\forJP L)$ is 0-perfect.
Since $L_{\pm 1}$ generates $\uTKK(\forJP L)$, $\psi$ is unique.
\end{proof}

\begin{cor}\label{0perf} A universal central $0$-extension of a
$0$-perfect Lie algebra is $0$-perfect.
\end{cor}

\begin{cor} \label{uTKKtoTKK} For every Jordan pair $P$ the algebra
$\uTKK(P)$ is a universal central 0-extension of $\TKK(P)$.
\end{cor}

\begin{thm} \label{TrecU} Assume that $\upsilon:U\to L$ is a central
0-extension of a 0-perfect Lie algebra $L$. This extension is universal
if and only if $U$ is 0-perfect and centrally $0$-closed.
\end{thm}

\begin{proof}
Assume first that $\upsilon:U\to L$ is universal and note that $U$ is
$0$-perfect by Corollary~\ref{0perf}.

Let $\varphi:K\to U$ be a central 0-extension of $U$, and consider the
subalgebra $K'=K_{-1}+[K_{-1},K_{1}]+K_{1}$ of $K$ together with the
inclusion homomorphism $\iota:K'\to K$. It is straightforward to check
that  $\upsilon\circ\varphi\circ\iota$ is an epimorphism and
$\Kr(\upsilon\circ\varphi\circ\iota)\sset K'_0$. Moreover,
$\varphi(k)\in \Kr(\upsilon)\subseteq \Cen(U)$ for every
$k\in\Kr(\upsilon\circ\varphi\circ\iota)$ and therefore
$[k,K_{\pm 1}]\in \Kr(\varphi)\cap K_{\pm 1}=0$. Thus
$\Kr(\upsilon\circ\varphi\circ\iota)\subseteq \Cen(K')$ and
$\upsilon\circ\varphi\circ\iota$ is a central 0-extension of $L$.

By universality of $\upsilon:U\to L$, there exists a map $\psi:U\to K'$
such that $(\upsilon\circ\varphi\circ\iota)\circ\psi=\upsilon$. It
follows that  $\varphi\circ\iota\circ\psi=\mbox{\rm id}_U$, so
$\iota\circ\psi$ is the splitting map for $\varphi$. Since $U$ is
$0$-perfect, the splitting map is unique.

Assume now that $U$ is 0-perfect and centrally 0-closed, and consider a
central 0-extension $\varphi:K\to L$.

To construct a map from $U$ to $K$ we consider the direct sum of
algebras $K\oplus U$ with the 3-grading $(K\oplus U)_i=K_i\oplus U_i$
and the graded subalgebra $A=\{(k,u):\varphi(k)=\upsilon(u)\}$.  It is
easy to see that the maps $\pi_K:A\to K$, $\pi_K(k,u)=k$ and $\pi_U:A\to
U$, $\pi_U(k,u)=u$ are graded epimorphisms, and that
$\varphi\circ \pi_K=\upsilon\circ\pi_U$.

We note that $\pi_U:A\to U$ is a central 0-extension of $U$ because
$\Kr(\pi_U)=\{(k,0):k\in\Kr(\varphi)\}\subseteq \Cen(A)\cap A_0$. Thus
there is a splitting morphism $\psi:U\to A$ with $\pi_U\circ
\psi=\mbox{\rm id}_U$. It follows that $\varphi\circ
(\pi_K\circ\psi)=(\upsilon\circ\pi_U)\circ\psi=\upsilon$.
The map $\pi_K\circ\psi$ is unique with this property since $U$ is
$0$-perfect.
\end{proof}

\begin{cor} \label{catequivJP} For a 3-graded Lie algebra $L$ the map
$\upsilon:\uTKK(\forJP L)\to L$, defined in Theorem~\ref{TuTKK}, is an
isomorphism if and only if $L$ is 0-perfect centrally 0-closed.

In particular, the category of Jordan Pairs $\JP$ is equivalent to the
category of 0-perfect centrally 0-closed 3-graded Lie algebras.
\end{cor}

\begin{proof} Recall that one has the full and faithful functor
$\uTKK:\JP\to \Lthr$. If $L$ is is 0-perfect and centrally $0$-closed,
then the map  $\upsilon:\uTKK(\forJP L)\to L$ is an isomorphism. The
converse follows immediately from Corollary~\ref{uTKKtoTKK} and
Theorem~\ref{TrecU}.
\end{proof}

\begin{cor} \label{catequivJTS}
The category of Jordan triple systems $\JTS$ is equivalent to the
category of 0-perfect centrally 0-closed 3-graded Lie algebras with
involution.
\end{cor}

\begin{proof} Let $L$ be a 0-perfect centrally 0-closed Lie algebra with
involution $\varepsilon$. Recall from Section~\ref{3grInv} that the
restriction $\widetilde\varepsilon$ of $\varepsilon$ onto the pair
$(L_{-1}, L_1)$ is an involution of the Jordan pair $\forJP L$, and
there exists an involutory isomorphism $\varphi$ from the Jordan pair
with involution $(\forJP L, \widetilde\varepsilon)$ onto $(\prJTS
L_1,\kappa)$ for the Jordan triple system $L_1=\forJTS(L,\varepsilon)$
and the canonical involution $\kappa$ on $\prJTS L_1$. Now one can
verify readily that $\uTKK\varphi$ is an isomorphism from
$(L,\varepsilon)$ onto $(\uTKK L_1,\kappa)$.
\end{proof}

\begin{cor} \label{catequivJA}
The category of unital Jordan algebras $\JA$ is equivalent to the
category of centrally closed ${\rm A}_1$-graded Lie algebras.
\end{cor}

\begin{proof} Let $L$ be a centrally closed Lie algebra with an ${\rm
A}_1$-grading. Then $L_1=\forJA L$ is a unital Jordan algebra and the
identity map $\id:L_1\to \forJA L$ lifts to the ${\rm A}_1$-graded
morphism $\widehat{\id}$ from $\uTKK L_1$ onto $L$. Since the kernel of
$\widehat{\id}$ is an ideal contained in $(\uTKK L_1)_0$ and hence
belongs to the center of $\uTKK L_1$, there is a splitting map for
$\widehat{\id}$ which is surjective since $\uTKK L_1$ is perfect.
\end{proof}
%\fi%%%%%%%%%%%%%%%%%%%%%%%%%%%%%%%%%%%%%%%%%%%%%%%%%%%%%%%%%%%%%%%%%%%%%%%%%%%%%%%%%%%%%%%%%%%%%%%%%%%%%%%%%

%\iffalse%%%%%%%%%%%%%%%%%%%%%%%%%%%%%%%%%%%%%%%%%%%%%%%%%%%%%%%%%%%%%%%%%%%%%%%%%%%%%%%%%%%%%%%%%%%%%%%%%%%
\subsection{Homological characterization of the extensions}
\label{homi}
The condition on extensions in Theorem~\ref{TrecU} has natural
homological characterizations which we establish next using a graded
version of the classical homology theory. Throughout this section we
assume that $k$ is a field of characteristic different from 2 and 3.

Recall that the homology groups $H_{*}(L)=H_{*}(L,k)$ of a Lie algebra
$L$ with coefficients in the one-dimensional trivial module $k$ can be
defined as the homology groups of the Chevalley-Eilenberg chain complex
$$
...\longrightarrow \wedge^{n+1} L \stackrel{\delta_n}{\longrightarrow}
\wedge^{n} L\longrightarrow ...
$$
where the boundary map $\delta_{n}$ is given by
\begin{eqnarray}\label{bo}
\delta_n\! (x_1\wedge x_2\wedge ...\wedge x_{n+1})\nonumber\\
=\sum_{1\leq i<j\leq n+1}(-1)^{i+j}[x_i,x_j]\wedge x_1\wedge ...\wedge
\widehat x_i\wedge ...\wedge \widehat x_j\wedge ...\wedge x_{n+1},
\end{eqnarray}
and $\widehat x_i$ indicates that $x_i$ is omitted.

The cohomology groups $H^*(L,M)$ of $L$ with coefficients in the any
trivial $L$-module $M$ are the cohomology groups of the dual cochain
complex
$$
...\longrightarrow\Hom _k(\wedge^n
L,M)\stackrel{\delta^n}{\longrightarrow}
\Hom _k(\wedge^{n+1} L,M)\longrightarrow ...
$$
with $\delta^n(f)=f\circ \delta_n$.

If $L$ is a graded algebra and $M$ is a graded module, one can introduce
graded versions of the complexes above by setting
$$(\wedge^{n} L)_i=\sum_{i_1+i_2+...+i_n=i} L_{i_1}\wedge
L_{i_2}\wedge...\wedge L_{i_n} \mbox{\ \ \ and}
$$ 
$$(\Hom _k(\wedge^n L,M))_i=\{f\in
\Hom _k(\wedge^n L,M):f(L_{i_1}\wedge L_{i_2}\wedge...\wedge
L_{i_n})\subseteq M_{i_1+i_2+...+i_n+i}\}$$ respectively and noting that
the maps $\delta_*$ and $\delta^*$ are graded. In this case, the
homology groups $H_{*}^{\rm gr}(L)$ of the graded Lie algebra $L$ are
the homology groups of the complex
$$
...\longrightarrow (\wedge^{n+1} L)_0
\stackrel{\delta_n}{\longrightarrow} (\wedge^{n} L)_0\longrightarrow
...
$$
and the cohomology groups $H^{*}_{\rm gr}(L,M)$ of the graded Lie
algebra $L$ are the cohomology groups of the complex
$$
...\longrightarrow(\Hom _k(\wedge^n
L,M))_0\stackrel{\delta^n}{\longrightarrow}
(\Hom _k(\wedge^{n+1} L,M))_0\longrightarrow ...
$$

\begin{thm} \label{homolal}
Assume that $L=L_{-1}\oplus L_0\oplus L_1$ is a $0$-perfect $3$-graded
Lie algebra  over a field $k$. Then following are equivalent.

\begin{itemize}
\item[{\rm (i)}] Every central 0-extension of $L$ splits.
\item[{\rm (ii)}] $H_2^{\rm gr}(L)=0$.
\item[{\rm (iii)}] $H^2_{\rm gr}(L,M)=0$ for every module $M$ such
    that $M=M_0$.
\end{itemize}
\end{thm}

\begin{proof}
Assume that (i) holds. Then every central 0-extension of $L$ splits in a
unique way since $L$ is $0$-perfect, so by Corollary~\ref{catequivJP}
the map $\upsilon=\widehat{(\mbox{\rm id}_{\forJP L})}$ is an
isomorphism.

Consider $\sum x_i\wedge y_i \in (\Kr(\delta_1))_0$. Since
$L_0=[L_{-1},L_1]$ implies that $L_0\wedge L_0\subseteq L_{-1}\wedge
L_1+{\rm Im}(\delta_2)$, we can assume that $x_i\in L_{-1}$ and $y_i\in
L_1$. Since $\sum [x_i,y_i]=0$ and $[L_{-1},L_1]\simeq \langle
L_{-1},L_1\rangle$, in $\langle L_{-1},L_1\rangle$ one has $\sum \langle
x_i,y_i\rangle=0$, so $\sum x_i\otimes y_i \in A(L_{-1}\otimes L_1)$. It
is only left to notice that identifying $L_{-1}\otimes L_1$ with
$L_{-1}\wedge L_1$, we have $A(L_{-1}\otimes L_1)\subseteq {\rm
Im}(\delta_2)$. Therefore $H_2^{\rm gr}(L)=0$.

Assume now that $H_2^{\rm gr}(L)=0$ and $M$ is an $L$-module with the
trivial grading $M=M_0$. Then $L_1M=L_1M_0\subseteq M_1=0$ and hence
$LM=([L_1,L_1]+L_1)M=0$. Thus $M$ is a trivial $L$-module.
If $\sigma:L\wedge L\to M$ is a 2-cocycle in $(\Kr(\delta^2))_0$, then
it is easy to see that $\Kr(\delta_1)\subseteq \Kr(\sigma)$. So one can
define $\tau_{\sigma}:L_0\to M_0$ by setting
$\tau_{\sigma}(x)=\sigma(\sum a_i\wedge b_i)$ for every $x=\sum
[a_i,b_i]\in L_0=[L_{-1},L_1]$. Clearly, $\sigma=\tau_{\sigma}\circ
\delta_1$, so $\sigma$ is a co-boundary. Hence assertion (ii) implies (iii).

Assume finally that $H^2_{\rm gr}(L,M)=0$ for every module $M$ such that
$M=M_0$. Let $\varphi:K\to L$ be a central 0-extension of $L$. Since
$\Kr(\varphi)\subseteq K_0$ and $L_0$ is projective as $k$-module, there is a linear
map $\eta:L\to K$ such that $\eta$ preserves the grading and
$\varphi\circ \eta=\id_L$.

We consider $\Kr(\varphi)$ as a trivial $L$-module with the grading
$\Kr(\varphi)=\Kr(\varphi)_0$. It is known that the map $\sigma:L\wedge
L\to \Kr(\varphi)$ defined by $\sigma(x\wedge
y)=[\eta(x),\eta(y)]-\eta([x,y])$ is a 2-cocycle. Moreover, for every
$x_i\in L_i$ and $y_j\in L_j$, we have $\sigma(x_i\wedge
y_j)=[\eta(x_i),\eta(y_j)]-\eta([x_i,y_j])\in M\cap K_{i+j}\subseteq
M_{i+j}$. Hence $\sigma\in Z^2_{\rm gr}(L,\Kr(\varphi))$.

Since $H^2_{\rm gr}(L,\Kr(\varphi))=0$, there is a linear map $\tau:L\to
\Kr(\varphi)$
such that $\sigma(x\wedge y)=\tau([x,y])$. Then the map $\psi=\eta+\tau$
is a Lie algebra homomorphism $\psi:L\to K$. Indeed,
\begin{eqnarray*}
\psi([x,y]))=\eta([x,y])+\tau([x,y])=[\eta(x),\eta(y)]=
[\eta(x)+\tau(x),\eta(y)+\tau(y)]\\
=[\psi(x),\psi(y)].
\end{eqnarray*}

Besides, $\varphi\circ\psi=\id_L$ since $\varphi\circ\tau=0$. Thus
$\psi$ is a splitting map for $\varphi$.
\end{proof}

The characterizations above enable us to restate
Corollaries~\ref{catequivJP} and \ref{catequivJTS} as follows:

\begin{theorema}
The category of Jordan pairs is equivalent to the category of $3$-graded
Lie algebras $L = L_{-1} \oplus L_0 \oplus L_1$ such that $L$ is
$0$-perfect and satisfies one of the equivalent conditions:
\begin{itemize}
\item[(i)]  $H_2^{\rm gr} (L) = 0$
\item[(ii)] $H^2 _{\rm gr} (L,M) = 0$ for every module $M$ with the
    trivial grading $M = M_0$.
\end{itemize}
\end{theorema}

\begin{theoremb}
The category of Jordan triple systems is equivalent to the category of
$3$-graded Lie algebras $L = L_{-1} \oplus L_0 \oplus L_1$ with
involution such that $L$ is $0$-perfect and satisfies one of the
equivalent conditions:
\begin{itemize}
\item[(i)]  $H_2^{\rm gr} (L) = 0$
\item[(ii)] $H^2 _{\rm gr} (L,M) = 0$ for every module $M$ with the
    trivial grading $M = M_0$.
\end{itemize}
\end{theoremb}

 Well-known facts on $H_2(L)$ and $H^2(L,M)$ and central extensions (see
 for example \cite[Chapter 7]{W}) together with
 Corollary~\ref{catequivJA} imply:

\begin{theoremc}
The category of unital Jordan algebras is equivalent to the category of
A$_1$-graded Lie algebras $L$ satisfying one of the equivalent
conditions:
\begin{itemize}
\item[(i)]  $H_2(L) = 0$
\item[(ii)] $H^2(L,M) = 0$ for every trivial module $M$.
\end{itemize}
\end{theoremc}

%\fi%%%%%%%%%%%%%%%%%%%%%%%%%%%%%%%%%%%%%%%%%%%%%%%%%%%%%%%%%%%%%%%%%%%%%%%%%%%%%%%%


\begin{thebibliography}{99}
\bibitem{ML} S.\,Mac\,Lane, ``Categories for the Working
    Mathematician'', Springer-Verlag, New York, 1998.
\bibitem{W} C.\,A.\,Weibel, ``An Introduction to Homological Algebra'',
    Cambridge Studies in Adv. Math., Cambridge Univ. Press, vol. 38,
    Cambridge, 1994.
\bibitem{AG} B.\,N.\,Allison and Y.\,Gao, {Central quotients and
coverings of Steinberg unitary Lie algebras}, Canad.  J.  Math. {48}
(1996) 449--482.
\bibitem{ABG2} B.\,N.\,Allison, G.\,Benkart, and Y.\,Gao, {Lie algebras
    graded by the root systems BC$_r$, $r\ge 2$}, Memoir Amer. Math.
    Soc. {158}, vol. 751, 2002.
\bibitem{B} W.\,Bertram, The geometry of Jordan and Lie structures,
    Lecture Notes in Mathematics, vol. 1754, Springer-Verlag, Berlin,
    2000.
\bibitem{BM} S.\,Berman and R.\,V.\,Moody, {Lie algebras graded by finite
    root systems and the intersection matrix algebras of Slodowy},
    Invent. Math.~{108} (1992) 323--347.
\bibitem{BN} W.\,Bertram and K.-H.\, Neeb, {Projective completions of
    Jordan pairs. Part I. The genaralized projective geometry of a Lie
    Algebra}, J. of Algebra {277} (2004) 474--519.
\bibitem{BS} G.\,Benkart and O.\,Smirnov, {Lie algebras graded by the
    root system BC$_1$}, Journal of Lie Theory~{13} (2003) 91--132.
\bibitem{BZ} G.\,Benkart and E.\,Zelmanov, {Lie algebras graded by
    finite root systems and the intersection matrix algebras}, Invent.
    Math.~{126} (1996) 1--45.
\bibitem{K} V.\,G.\,Kac, {Some Remarks on Nilpotent Orbits}, J. of
    Algebra {64} (1980) 190--213.
\bibitem{L} O.\,Loos, {Jordan Pairs},
Lecture Notes in Math., vol.460, Springer-Verlag, New York, 1975.
\bibitem{M} K.\,McCrimmon, {A taste of Jordan algebras},
    Springer-Verlag, New York, 2004.
\bibitem{S} O.\,Smirnov, {Imbedding of Lie triple systems into Lie
    algebras}, submitted, J. of Algebra (preprint on arXiv:0906.1170).
\bibitem{T} S.\,Tan, {TKK algebras ond Vertex Operator representations},
    J. of Algebra {211} (1999) 298--342.
\bibitem{ZSSS} K.\,A.\,Zhevlakov, A.\,M.\,Slin'ko, I.\,P.\,Shestakov,
    A.\,I.\,Shirshov, ``Rings that are nearly associative'', Academic
    Press, New York, 1982.
\end{thebibliography}
\end{document}